\documentclass{sl}     
\usepackage[latin1]{inputenc}
\usepackage{slsec}     
\usepackage{stud_log} 
\usepackage{slfoot}
\usepackage{slthm}     
                         
\usepackage{amsmath}
\usepackage{amssymb}

\usepackage{graphicx}

\usepackage{enumerate}

\def\kn{\kern.1em}

\renewcommand{\theequation}{\Roman{equation}}

\newtheorem{theorem}{Theorem}[section]  
\newtheorem{corollary}[theorem]{Corollary}
\newtheorem{lemma}{Lemma}[section]
\newtheorem{proposition}{Proposition}[section]

\theoremstyle{definition}
\newtheorem{definition}[theorem]{Definition}
\newtheorem{remark}{Remark}[section]

\HeadingsInfo{G. Pelaitay, W. Zuluaga}{Tense operators on distributive lattices with implication}

\begin{document}

\setcounter{page}{1}     

 




\twoAuthorsTitle{G.\ts Pelaitay}{W. 
\ts Zuluaga}{Tense operators on distributive lattices with implication}


   


\begin{abstract} Inspired by the definition of tense operators on distributive lattices presented by Chajda and Paseka in 2015, in this paper, we introduce and study the variety of tense distributive lattices with implication and we prove that these are categorically equivalent to a full subcategory of the category of tense centered Kleene algebras with implication. Moreover, we apply such an equivalence to describe the congruences of the algebras of each variety by means of tense 1-filters and tense centered deductive systems, respectively.
\end{abstract}


\section[Introduction]{Introduction}

Intuitionistic tense logic was introduced by Ewald in \cite{Ewald} as a generalization of Prior's classical tense logic \cite{P}. The corresponding logical system was called IKt and its language is the language of intuitionistic logic expanded by four unary connectives $G,H, F$ and $P$. Such a system is axiomatized as follows:
\begin{itemize}
  \item []
\begin{itemize}
\item [(A1)]  All axioms of intuitionistic logic (Int).
\item [(A2)]  $G(\alpha\to\beta)\to (G\alpha\to G\beta)$,\, $H(\alpha\to\beta)\to (H\alpha\to H\beta)$,
\item [(A3)]  $G(\alpha\wedge \beta)\leftrightarrow G\alpha\wedge G\beta$,\, $H(\alpha\wedge \beta)\leftrightarrow H\alpha\wedge H\beta$,
\item [(A4)]  $F(\alpha\vee\beta)\leftrightarrow F\alpha\vee F\beta$,\, $P(\alpha\vee\beta)\leftrightarrow P\alpha\vee P\beta$,  
\item [(A5)]  $G(\alpha\to \beta)\to (F\alpha\to F\beta)$,\, $H(\alpha\to\beta)\to (P\alpha\to P\beta)$,
\item [(A6)] $G\alpha\wedge F\beta\to F(\alpha\wedge \beta)$,\,  $H\alpha\wedge P\beta \to P(\alpha\wedge\beta)$,
\item [(A7)]  $G\neg\alpha\to \neg F\alpha$,\,  $H\neg\alpha\to \neg P\alpha$,
\item [(A8)]  $FH\alpha\to \alpha$,\,  $PG\alpha\to \alpha$,
\item [(A9)]  $\alpha\to GP\alpha$,\,  $\alpha\to HF\alpha$,
\item [(A10)]  $(F\alpha\to G\beta)\to G(\alpha\to \beta)$,\, $(P\alpha\to H\beta)\to H(\alpha\to \beta)$, 
\item [(A11)] $F(\alpha\to \beta)\to (G\alpha\to F\beta)$,\, $P(\alpha\to \beta)\to (H\alpha\to P\beta)$.
\end{itemize}
\item[]        
\end{itemize}
The inference rules are modus ponens and the following necessitation rules:
\begin{itemize}
\item[](RG)\,\, $\displaystyle\frac{\alpha}{G\alpha}$,\hspace{3cm}(RH)\,\, $\displaystyle\frac{\alpha}{H\alpha}.$
\end{itemize}

It is worth mentioning that this axiomatization is not independent, due to some of the axioms can be deduced from the remaining ones. Considering that an algebraic axiomatization of intutionistic logic can be provided via Heyting algebras, in \cite{FP} Figallo and Pelaitay introduced the variety IKt of
IKt-algebras and proved that the IKt system has IKt-algebras as an equivalent algebraic semantics. Afterwards, the notion of tense operators on bounded distributive lattices was introduced by Chajda and Paseka in \cite{CP}. In this paper and also in \cite{Menni}, it was revealed that there is a close relationship between tense distributive lattices and IKt-algebras. Following that direction, it is natural to take the study of tense operators to more general structures than Heyting algebras. This is the case of algebras with implication, also known as DLI-algebras \cite{Celani04}.

Encouraged by the results of Kalman in \cite{Kalman}, R. Cignoli in \cite{Cignoli} proved that
Kalman's construction induces a functor K from the category of bounded
distributive lattices into the category of centered Kleene algebras. It was also
shown in \cite{Cignoli} that K has a left adjoint \cite[Theorem 1.7]{Cignoli} and moreover, that this adjunction determined an equivalence between the category of bounded distributive lattices and the
full subcategory of centered Kleene algebras with the interpolation property \cite[Theorem 2.4]{Cignoli}. In particular, it was shown that
there exists an equivalence between the category of Heyting algebras and the
category of centered Nelson algebras \cite[Theorem 3.14]{Cignoli}. Later, in \cite{CCS}, these results were extended to the context of DLI-algebras.

Motivated by the latter in this paper we consider tense operators on DLI-algebras and extend the equivalence developed in \cite{CCS} to the tense case.
\\

The paper is organized as follows. In Section \ref{s1}, we present the basics required to read this work. In Section \ref{s2} and \ref{s4}, we introduce and we prove some useful properties about tense DLI$^{+}$-algebras and tense Kleene algebras with implication, respectively. In Section \ref{s5} we establish a categorical equivalence between the categories of tense DLI$^{+}$-algebras such that $G(0)=0$ and $H(0)=0$  and the category of tense centered Kleene algebras satisfying the condition (CK) (Theorem \ref{equivalence theorem}). In Section \ref{s6} we characterize the congruences of tense DLI$^{+}$-algebras such that $x\rightarrow x=1$ by means of tense 1-filters. The contents of Section \ref{s7} are devoted to obtain an extrinsical and intrinsical characterization of the congruences of tense centered Kleene algebras satisfying the equation $x\Rightarrow x=1$. The extrinsical description is provided in subsection \ref{congruence tense 1-filters} in terms of the 1-filters of the associated tense DLI$^{+}$-algebras and later in subsection \ref{Centered tense deductive systems} we get the intrinsical description via centered tense deductive systems. Finally, in Section \ref{s8} we apply the results we obtained to the particular class of tense Heyting algebras.

\section[Preliminaries]{Preliminaries}\label{s1}

In this paper, we take for granted that the reader is familiar with the standard concepts and results on distributive lattices, Heyting algebras, category theory and universal algebra as presented in \cite{Balbes,Burris,CHK,MacLane}. However, for the sake of readability, in this section we shall summarize the fundamental concepts we use along this paper.

If $\mathcal{K}$ denotes a class of algebras of a given type and $\mathbf{A}\in \mathcal{K}$, then we write $\mathsf{Con}_{\mathcal{K}}(\mathbf{A})$ for the lattice of congruences of $\mathbf{A}$. If $\theta \in \mathsf{Con}_{\mathcal{K}}(\mathbf{A})$, we also write $a\theta b$ to denote $(a,b)\in \theta$. Finally, if the class of algebras is clear from the context and no clarification is needed, we write $\mathsf{Con}(\mathbf{A})$ instead of $\mathsf{Con}_{\mathcal{K}}(\mathbf{A})$.

A De Morgan algebra is an algebra $\langle A,\wedge,\vee,\sim,0,1\rangle$ of
type $(2, 2, 1, 0, 0)$ such that $\langle A, \wedge, \vee, 0, 1\rangle$ is a bounded distributive lattice and $\sim$ fulfills the equations $\sim\sim x=x$ and $\sim (x\vee y)=\sim x\wedge \sim y$. A Kleene algebra is a De Morgan algebra in which the inequality $x\wedge \sim x\leq y\vee \sim y$ holds. A centered Kleene algebra is a Kleene algebra with an element $c$ such that $c = \sim c$. It follows from the distributivity of the lattice that $c$ is necessarily unique.

A quasi--Nelson algebra is a Kleene algebra such that for each pair $x,y,$ there exists $x\to (\sim x\vee y),$ where $\to$ is the Heyting implication \cite{V,R}. We define the binary operation $\Rightarrow$ as $x\Rightarrow y=x\to (\sim x\vee y)$. A Nelson algebra is a quasi--Nelson algebra satisfying the equation $(x\wedge y)\Rightarrow z=x\Rightarrow (y\Rightarrow z)$. A quasi--Nelson algebra is a Nelson algebra if and only if it satisfies the interpolation property \cite[Theorem  3.5]{Cignoli}. We say that a Nelson algebra $\langle A,\wedge,\vee,\Rightarrow, \sim,c,0,1\rangle$  is a centered Nelson algebra if the reduct $\langle A,\wedge,\vee, \sim,c,0,1\rangle$ is a centered Kleene algebra.

\section[Tense operators on DLI$^{+}$-algebras]{Tense operators on DLI$^{+}$-algebras}\label{s2}

In this section, we will define the tense DLI$^{+}$-algebras and we will prove some basic properties.

Recall from \cite{Celani04} that an algebra $\mathcal{A}=\langle A, \wedge, \vee, \rightarrow, 0, 1\rangle$ of type $(2, 2, 2, 0, 0)$ is a DLI-algebra if $\langle A, \wedge, \vee, 0, 1\rangle$ is a bounded distributive lattice and the following conditions are satisfied:
\begin{itemize}
    \item[(I1)] $(a\rightarrow b)\wedge (a\rightarrow d)=a\rightarrow (b\wedge d)$,
    \item[(I2)] $(a\rightarrow d)\wedge (b\rightarrow d)=(a\vee b)\rightarrow d$,
    \item[(I3)] $0\rightarrow a=1$,
    \item[(I4)] $a\rightarrow 1=1$.
\end{itemize}

The following Lemma, which was originally proved in Remark 3.2 of \cite{CCS}, provides some monotonicity properties of the implication in DLI-algebras.

\begin{lemma}\label{monotonia}
Let $\mathbf{A}=\langle A, \wedge, \vee, \rightarrow, 0, 1\rangle$ be a DLI-algebra and let $x,y,z\in A$. Then, if $x\leq y$, it is the case that $y\rightarrow z\leq x\rightarrow z$ and $z\rightarrow x\leq z\rightarrow y$.
\end{lemma}

We write $\mathrm{DLI}^{+}$ for the variety of DLI-algebras whose algebras satisfy the following equation:
\begin{itemize}
\item[(I5)] $a\wedge (a\rightarrow b)\leq b$.
\end{itemize}

\begin{definition}\label{tDLI} Let $\mathbf{A}=\langle A,\vee,\wedge,\to,0,1\rangle$ be a DLI$^{+}$-algebra, let $G$, $H$, $F$ and $P$ be unary operations on $A$ satisfying:

\begin{itemize}
    \item [(T1)] $P(x)\leq y$ if and only if $x\leq G(y),$
    \item [(T2)] $F(x)\leq y$ if and only if $x\leq H(y),$
    \item [(T3)] $G(x)\wedge F(y)\leq F(x\wedge y)$ and $H(x)\wedge P(y)\leq P(x\wedge y),$
    \item [(T4)] $G(x\vee y)\leq G(x)\vee F(y)$ and $H(x\vee y)\leq H(x)\vee P(y),$
    \item [(T5)] $G(x\rightarrow y)\leq G(x)\rightarrow G(y)$ and $H(x\rightarrow y)\leq H(x)\rightarrow H(y),$
    \item [(T6)] $G(x\to y)\leq F(x)\to F(y)$ and $H(x\to y)\leq P(x)\to P(y)$.
\end{itemize}
Then the algebra $\mathbf{L}=(\mathbf{A},G,H,F,P)$ will be called tense DLI$^{+}$-algebra and $G$, $H$, $F$ and $P$ will be called
tense operators.
\end{definition}

\begin{remark} Let $\mathbf{L}=(\mathbf{A},G,H,F,P)$ be a tense DLI$^{+}$-algebra. Then, from (T1) to (T4), we have that the reduct $\langle A,\vee,\wedge,G,H,F,P\rangle$ is a tense distributive lattice (see \cite[Definition 3.10]{Menni}, \cite{CP}).
\end{remark}

We will list some basic properties valid in tense DLI$^{+}$-algebras, proving just some of them.

\begin{proposition} Let $\mathbf{L}$ be a tense DLI$^{+}$-algebra. Then

\begin{itemize}
    \item [{\rm (T7)}] $G(1)=1$ and $H(1)=1,$
    \item [{\rm (T8)}] $G(x\wedge y)=G(x)\wedge G(y)$ and $H(x\wedge y)=H(x)\wedge H(y),$
    \item [{\rm (T9)}] $x\leq GP(x)$ and $x\leq HF(x),$
    \item [{\rm (T10)}] $F(0)=0$ and $P(0)=0,$
    \item [{\rm (T11)}] $F(x\vee y)=F(x)\vee F(y)$ and $P(x\vee y)=P(x)\vee P(y),$
    \item [{\rm (T12)}] $FH(x)\leq x$ and $PG(x)\leq x,$
    \item [{\rm (T13)}] $x\leq y$ implies $G(x)\leq G(y)$ and $H(x)\leq H(y),$
    \item [{\rm (T14)}] $x\leq y$ implies $F(x)\leq F(y)$ and $P(x)\leq P(y),$
    \item [{\rm (T15)}] $x\wedge F(y)\leq F(P(x)\wedge y)$ and $x\wedge P(y)\leq P(F(x)\wedge y),$
    \item [{\rm (T16)}] $F(x)\wedge y=0$ if and only if $x\wedge P(y)=0,$
    \item [{\rm (T17)}] $G(x\vee H(y))\leq G(x)\vee y$  and $H(x\vee G(y))\leq H(x)\vee y,$
    \item [{\rm (T18)}] $x\vee H(y)=1$ if and only if $G(x)\vee y=1$.
\end{itemize}

\end{proposition}

\begin{proof} Notice that (T7) to (T12) follows from (T1) and (T2). Axioms (T13) and (T14) are a consequence of axioms (T8) and (T11), respectively. Next, let us prove (T15). From (T9), we have that $x\wedge F(y)\leq GP(x)\wedge F(y)$. From this statement and from (T3), we obtain $x\wedge F(y)\leq F(P(x)\wedge y)$.  The other inequality is analogous. Let us verify (T16). Let us assume that $F(x)\wedge y=0$. From (T10) and (T15), we obtain $x\wedge P(y)\leq P(F(x)\wedge y)=P(0)=0$. Hence, $x\wedge P(y)=0$. Similarly, we can prove the other direction. Finally, axioms (T17) and (T18) can be proved using a similar technique to that used in the proof of property (T15) and (T16), respectively.
\end{proof}

\section{Tense operators on Kleene algebras with implication}\label{s4}

Kleene algebra is a De Morgan algebra in which the inequality $x\wedge \sim x\leq y\vee \sim y $ holds. A centered Kleene algebra is a Kleene algebra with an element $c$ such that $c = \sim c$. It follows from the distributivity of the lattice that $c$ is necessarily unique. Recall from \cite{CCS} that an algebra $\mathbf{T}=\langle T, \wedge, \vee, \Rightarrow, \sim, c, 0, 1\rangle$ of type $(2, 2, 2, 1, 0, 0, 0)$ is a Kleene algebra with implication (KI-algebra for short) if  $\langle T, \wedge, \vee, \sim, c, 0, 1\rangle$ is a centered Kleene algebra and for every $x,y\in T$ the following conditions hold: 
\begin{itemize}
    \item[(KI1)] $\langle T,\vee, \wedge,\Rightarrow, 0,1 \rangle$ is a DLI-algebra.
    \item[(KI2)] $(x\wedge (x\Rightarrow y))\vee c\leq y\vee c$.
    \item[(KI3)] $c\Rightarrow c=1$. 
    \item[(KI4)] $(x\Rightarrow y)\wedge c=(\sim x\vee y)\wedge c$.
    \item[(KI5)] $(x\Rightarrow \sim y)\vee c=(x\Rightarrow (\sim y \vee c))\wedge (y\Rightarrow (\sim x \vee c))$.
\end{itemize}

\begin{definition}\label{def: tKI-algebras}
An algebra $\mathbf{U}=(\mathbf{T}, G,H)$ is a tense Kleene algebra with implication (tense KI-algebra for short) if $\mathbf{T}=\langle T, \wedge, \vee, \Rightarrow, \sim, c, 0, 1\rangle$ is a KI-algebra and the following conditions hold:
\begin{itemize}
   \item[{\rm (t1)}] $G(1)=1$ and $H(1)=1,$
  \item[{\rm (t2)}] $G(x\wedge y)=G(x)\wedge G(y)$ and $H(x\wedge y)=H(x)\wedge H(y),$
  \item[{\rm (t3)}] $x\leq GP(x)$ and $x\leq HF(x),$ where $P(x):=\sim H(\sim x)$ and $F(x):=\sim G(\sim x),$
  \item[{\rm (t4)}] $G(x\vee y)\leq G(x)\vee F(y)$ and $H(x\vee y)\leq H(x)\vee P(y),$   
  \item[{\rm (t5)}] $G(x\Rightarrow y)\leq G(x)\Rightarrow G(y)$ and $H(x\Rightarrow y)\leq H(x)\Rightarrow H(y),$
  \item[{\rm (t6)}] $G(x\Rightarrow y)\leq F(x)\Rightarrow F(y)$ and $H(x\Rightarrow y)\leq P(x)\Rightarrow P(y)$.
\end{itemize}
Moreover, if $G(c)=c=H(c)$, then we say that $\mathbf{U}$ is a tense centered KI-algebra.
\end{definition}

\begin{remark} Let $\mathbf{U}=(\textbf{T},G,H)$ be a tense KI-algebra. Then, from (t1) to (t4), we have that the reduct $\langle T,\vee,\wedge,\sim,G,H,0,1\rangle$ is a tense De Morgan algebra (see \cite{FP14},\cite{CP}).
\end{remark}

The following proposition contains some properties of tense KI-algebras which will be useful along the paper.

\begin{proposition}
Let $\mathbf{U}$ be a tense KI-algebra. Then, the following hold:

\begin{itemize}
    \item [{\rm (t7)}] $F(0)=0$ and $P(0)=0,$
    \item [{\rm (t8)}] $F(x\vee y)=F(x)\vee F(y)$ and $P(x\vee y)=P(x)\vee P(y),$
    \item [{\rm (t9)}] $PG(x)\leq x$ and $FH(x)\leq x,$
\item [{\rm (t10)}] $x\leq y$ implies $G(x)\leq G(y)$ and $H(x)\leq H(y),$
    \item [{\rm (t11)}] $x\leq y$ implies $F(x)\leq F(y)$ and $P(x)\leq P(y),$
    \item [{\rm (t12)}] $G(x)\wedge F(y)\leq F(x\wedge y)$ and $H(x)\wedge P(y)\leq P(x\wedge y),$
    \item [{\rm (t13)}] $x\wedge F(y)\leq F(P(x)\wedge y)$ and $x\wedge P(y)\leq P(F(x)\wedge y),$
    \item [{\rm (t14)}] $F(x)\wedge y=0$ if and only if $x\wedge P(y)=0,$
    \item [{\rm (t15)}] $G(x\vee H(y))\leq G(x)\vee y$  and $H(x\vee G(y))\leq H(x)\vee y,$
    \item [{\rm (t16)}] $x\vee H(y)=1$ if and only if $G(x)\vee y=1$.
\end{itemize}

\end{proposition}

In what follows, we give some algebraic properties of the variety of tense centered KI-algebras which we will use later.

\begin{proposition}\label{propiedades tKLc}
In every tense centered KI-algebra $\mathbf{U}$ the following properties hold:

\begin{itemize}
    \item [{\rm (c1)}] $F(c)=c$ and $P(c)=c,$
    \item [{\rm (c2)}] $G(x\vee c)=G(x)\vee c$ and $H(x\vee c)=H(x)\vee c,$
    \item [{\rm (c3)}] $F(x\wedge c)=F(x)\wedge c$ and $P(x\wedge c)=P(x)\wedge c.$
\end{itemize}

\end{proposition}

\begin{proof} We will only prove (c1) and (c2). 

\noindent (c1): From (t3), we have that $F(c)=\sim G(\sim c)=\sim G(c)=\sim c=c$. Similarly, $P(c)=c$.

\noindent (c2): Since $x\leq x\vee c$ and $c\leq x\vee c,$ by (t10), we have that $G(x)\leq G(x\vee c)$ and $c=G(c)\leq G(x\vee c)$. So, $G(x)\vee c\leq G(x\vee c).$ On the other hand, from (t4) and (c1), we obtain $G(x\vee c)\leq G(x)\vee F(c)=G(x)\vee c.$ Therefore, $G(x\vee c)=G(x)\vee c$. Similarly, $H(x\vee c)=H(x)\vee c.$
\end{proof}

\section{Kalman's construction}\label{s5}

In this section, we prove some results that establish the connection between tense DLI$^{+}$-algebras and tense KI-algebras.

Let $\mathbf{L}=(\mathbf{A},G,H,F,P)$ be a tense DLI$^{+}$-algebra and let us consider

$$K(A):=\{(a,b)\in A\times A: a\wedge b=0\}.$$

It is well known from \cite[Proposition 4.3]{CCS} that by defining: 
\begin{eqnarray*}
  (a,b)\vee (x,y) & = & (a\vee x,b\wedge y),
  \\ 
  (a,b)\wedge (x,y) & = &(a\wedge x,b\vee y),
  \\
  (a,b) \Rightarrow (x,y) & = & ((a\rightarrow x)\wedge (y\rightarrow b),a\wedge y)
  \\
  \sim (a,b)& = &(b,a),
  \\
   0 &= &(0,1),
   \\
   1 &= &(1,0)
   \\
   c &= &(0,0)
\end{eqnarray*}
we get that the algebra $\mathbf{A}_{K}=\langle K(A),\vee,\wedge,\sim, \Rightarrow,c,0,1\rangle$ is a  KI-algebra. Now, we define on $K(A)$ the following unary operators: 
\begin{eqnarray*}
  G_K((a,b)) & = & (G(a),F(b)),
  \\
  H_K((a,b)) & = & (H(a),P(b)),
  \\
  F_{K}((a,b)) & = & (F(a),G(b)),
  \\
  P_{K}((a,b)) & = & (P(a),H(b)).
\end{eqnarray*}

\begin{lemma}\label{lf} Let $\mathbf{L}=(\mathbf{A},G,H,F,P)$ be a tense DLI$^{+}$-algebra and let $(a,b)\in K(A)$. Then, the following hold:
\begin{itemize}
    \item [{\rm (a)}] $G_{K}(a,b), H_{K}(a,b)\in K(A)$, 
    \item [{\rm (b)}] $F_{K}(a,b)=\sim G_{K}(\sim (a,b))$ and $P_{K}(a,b)=\sim H_{K}(\sim (a,b))$,
    \item [{\rm (c)}] $F_{K}(a,b), P_{K}(a,b)\in K(A)$.
\end{itemize}

\end{lemma}

\begin{proof} We will only prove (a). Let $(a,b)\in K(A)$. Hence, $a\wedge b=0$. Then, from (T3) and (T10), $G(a)\wedge F(b)\leq F(a\wedge b)=F(0)=0.$ Therefore, $(G(a),F(b))\in K(A)$. In a similar way, we can prove $H_{K}(a,b)\in K(A)$.
\end{proof}

\begin{lemma}\label{kalmanH} Let $\mathbf{L}=(\mathbf{A},G,H,F,P)$ be a tense DLI$^{+}$-algebra. Then, $$K(\mathbf{L})=(\mathbf{A}_{K},G_{K},H_{K})$$ is a tense KI-algebra. Moreover, if $G(0)=0$ and $H(0)=0$, then, $K(\mathbf{L})$ is a tense centered KI-algebra.
\end{lemma}

\begin{proof} From Lemma \ref{lf}, it turns out that the operators
$G_K$ and $H_K$ are well defined. Now, we will prove the axioms of tense
KI-algebras. Due to the symmetry of
tense operators $G$ and $H$, we will prove the axioms only for the operator $G$. Let $(a,b), (x,y)\in K(A)$.

\vspace{2mm}

\noindent (t1): From (T7) and (T10), we have $G_{K}(1,0)=(G(1),F(0))=(1,0)=1$.

\vspace{2mm}

\noindent (t2): From (T8) and (T11), we have $G_{K}((a,b)\wedge (x,y))=G_{K}(a\wedge x, b\vee y)=(G(a\wedge x),F(b\vee y))=(G(a)\wedge G(x),F(b)\vee F(y))=(G(a),F(b))\wedge (G(x),F(y))=G_{K}(a,b)\wedge G_{K}(x,y).$

\vspace{2mm}

\noindent  (t3): From properties (T9) and (T12), we have $(a,b)\wedge G_{K}P_{K}(a,b)=(a,b)\wedge G_{K}((P(a),H(b)))=(a,b)\wedge (GP(a),FH(b))=(a\wedge GP(a),b\vee FH(b))=(a,b).$

\noindent (t4): From (T3) and (T4), we have $G_{K}(a\vee x, b\wedge y)=(G(a\vee x), F(b\wedge y))\leq (G(a)\vee F(x),G(b)\wedge F(y))=(G(a),G(b))\vee (F(x),F(y))=G_{K}(a,b)\vee F_{K}(x,y)$.

\noindent (t5):  From (T5) and (T6), we have $G(a\to x)\leq G(a)\to G(x)$ and $G(y\to b)\leq F(y)\to F(b)$. Besides, from (T3), $G(a)\wedge F(y)\leq F(a\wedge y)$. Hence, $(G(a\to x)\wedge G(y\to b),F(a\wedge y))\leq ((G(a)\to G(x))\wedge (F(y)\to F(b)),G(a)\wedge F(y)).$ Therefore, $G_{K}((a,b)\Rightarrow (x,y))\leq G_K(a,b)\Rightarrow G_K(x,y).$ 

\noindent (t6): From (T3) and (T4), we obtain $G(a\vee x)\leq G(a)\vee F(x)$ and $F(b)\wedge G(y)\leq F(b\wedge y).$ Hence, $(G(a\vee x),F(b\wedge y))\leq (G(a)\vee F(x), F(b)\wedge G(y)).$ Therefore, $G_{K}((a,b)\Rightarrow (x,y))\leq F_K(a,b)\Rightarrow F_K(x,y).$

For the moreover part, notice that from Definition \ref{def: tKI-algebras}, we only have to check $G_K(c)=c$ and $H_K(c)=c$. From the hypothesis and (T10), we have 
\[G_K(c)=G_K(0,0)=(G(0),F(0))=(0,0)=c.\]
Similarly, $H_{K}(c)=c$. This concludes the proof.

\end{proof}



We write $\mathsf{tDLI}^{+}_{0}$ for the category whose objects are tense DLI$^{+}$-algebras which satisfy the additional condition $G(0)=0=H(0)$  and $\mathsf{tKI}_{c}$ for the category whose objects are tense centered KI-algebras. In both cases, the morphisms are the corresponding algebra homomorphisms. Moreover, if $\mathbf{L}=(\mathbf{A},G,H,F,P)$ and $\mathbf{M}=(\mathbf{B},G,H,F,P)$ are DLI$^{+}$-algebras and $f : \mathbf{L} \rightarrow \mathbf{M}$ is a morphism in $\mathsf{tDLI}^{+}_{0}$, then it is no hard to see that the map $K(f): K(A) \rightarrow K(B)$
given by $K(f)(x,y) = (f(x),f(y))$ is a morphism in $\mathsf{tKI}_{c}$ from $(K(\mathbf{L}),G_{K},H_{K})$ to $(K(\mathbf{M}),G_{K},H_{K})$. It is clear that these assignments establish a functor K from $\mathsf{tDLI}^{+}_{0}$ to $\mathsf{tKI}_{c}$.

\vspace{2mm}

Let $\mathbf{T}=\langle T,\wedge,\vee,\Rightarrow,\sim,c,0,1\rangle$ be a KI-algebra and define $$C(T):=\{x\in T: x\geq c\}.$$

Then, the structure $$\mathbf{T}_C=\langle C(T),\wedge,\vee,\Rightarrow,c,1\rangle$$ is a DLI$^{+}$-algebra. Moreover, if $f:\mathbf{T}\rightarrow \mathbf{S}$ is a homomorphism KI-algebras, then it follows that $C(f):C(T)\rightarrow C(S),$ defined by $C(f)(x)=f(x)$, is a homomorphism of DLI$^{+}$ algebras (see \cite[Proposition 4.7]{CCS}).

\begin{lemma}\label{tDLI+0 to tKIc} 
Let $(\mathbf{T},G,H)$ be a tense centered KI-algebra. Then, $$C(\mathbf{T})=(\mathbf{T}_C,G,H,F,P)$$ is a tense DLI$^{+}_{0}$-algebra. Moreover, if $f:(\mathbf{T},G,H)\rightarrow (\mathbf{S},G,H)$ is a morphism in $\mathsf{tKI}_{c}$, then $C(f)$ is a morphism in $\mathsf{tDLI}^{+}_{0}$.
\end{lemma}

\begin{proof} It is routine.
\end{proof}

Let $f:(\mathbf{T},G,H)\rightarrow (\mathbf{S},G,H)$ be a homomorphisms of tense centered KI-algebras. It is clear now, from Lemma \ref{tDLI+0 to tKIc}, that the assingments $\mathbf{T}\mapsto C(\mathbf{T})$, $f\mapsto C(f)$ determine a functor C from $\mathsf{tKI}_{c}$ to $\mathsf{tDLI}^{+}_{0}$.

\begin{lemma} Let $\mathbf{L}=(\mathbf{A},G,H,F,P)$ be a tense DLI$_{0}^{+}$-algebra. Then, the map $\alpha_{\mathbf{L}}:\mathbf{L}\rightarrow C(K(\mathbf{L}))$ given by $\alpha_{\mathbf{L}}(x,y)=(x,0)$ is an isomorphism in $\mathsf{tDLI}^{+}_{0}$.
\end{lemma}

\begin{proof} Taking into account \cite[Remark 2.2.]{CCS}, we only have to prove that $\alpha_{\mathbf{L}}$ preserves the tense operators. Let $x\in A$. Then,
\begin{itemize}
  \item $\alpha_{\mathbf{L}}(G(x))=(G(x),0)=(G(x),F(0))=G_K((x,0))=G_K(\alpha_{\mathbf{L}}(x))$.
  \item $\alpha_{\mathbf{L}}=(F(x),0)=(F(x),G(0))=F_K((x,0))=F_K(\alpha_{\mathbf{L}}(x))$.
\end{itemize}
In a similar fashion, we can prove $\alpha_{\mathbf{L}}(H(x))=H_{K}(\alpha_{\mathbf{L}}(x))$ and $\alpha_{\mathbf{L}}(P(x))=P_{K}(\alpha_{\mathbf{L}}(x))$.
\end{proof}

If $\mathbf{U}=(\mathbf{T},G,H)$ is a centered KI-algebra, then $$\beta_{\mathbf{U}}:\mathbf{U}\rightarrow K(C(\mathbf{U}))$$ given by  $\beta_{\mathbf{U}}(x)=(x\vee c, \sim x\vee c)$ is an injective homomorphism of centered KI-algebras \cite[Proposition 4.9.]{CCS}. Now we consider the folllowing condition on centered KI-algebras.

\begin{itemize}
    \item [(CK)] For every $x,y\geq c$ such that $x\wedge y\leq  c,$ there exists $z\in T$ such that $z\vee c=c$ and $\sim z\vee c=y.$
\end{itemize}

Then, the following lemma follows from the definition of $\beta_{\mathbf{U}}$.

\begin{lemma} \textnormal{\cite[Remark 2.6.]{CCS}} Let $\mathbf{U}$ be a centered KI-algebra. Then, $\mathbf{U}$ satisfies {\rm (CK)} if and only if $\beta_{\mathbf{U}}$ is a surjective map. 
\end{lemma}

Next, we extend to the tense case, the result given in \cite[Remark 2.2.]{CCS}.

\begin{lemma} Let $\mathbf{U}=(\mathbf{T},G,H)$ be a tense centered KI-algebra. Then, the map $\beta_{\mathbf{U}}$ is a monomorphism in $\mathsf{tKI}_{c}$.
\end{lemma}

\begin{proof} Let $x\in T$.  From (t3), (c2), (c1) and (t8), we have that $\beta_{\mathbf{U}}(G(x))=(G(x)\vee c, \sim G(x)\vee c)=(G(x)\vee c, F(\sim x)\vee c)=(G(x\vee c), F(\sim x\vee c))=G_{K}((x\vee c,\sim x\vee c))$. Similarly, $\beta_{\mathbf{U}}(H(x))=H_{K}(\beta_{\mathbf{U}}(x))$.
\end{proof}

We will denote by $\mathsf{itKI}_{c}$ for the full subcategory of $\mathsf{tKI}_{c}$ whose objects satisfy the condition (CK). 

Straightforward computations based on previous results of this section prove
the following result.

\begin{theorem}\label{equivalence theorem}
The functors K and C establish a categorical equivalence between $\mathsf{tDLI}^{+}_{0}$ and $\mathsf{itKI}_{c}$ with natural isomorphisms $\alpha$ and $\beta$.
\end{theorem}

\section{Congruences of $\mathrm{tDLI}_{01}^{+}$}\label{s6}

Let $\mathbf{A} \in \mathrm{DLI}^{+}$. Recall from \cite{CCS}, that $\mathbf{A}$ is a DLI$_{1}^{+}$-algebra if for all $x\in A$, $x\rightarrow x=1$. In addition, a lattice filter $S$ of $\mathbf{A}$ is called \emph{1-filter} if
\[((a\wedge f)\rightarrow b)) \rightarrow (a\rightarrow b)\in S,\]
for every $a,b\in S$ and $f\in S$. In Proposition 5.7 and Lemma 5.10 of \cite{CCS} it was proved that there is bijection between $\mathsf{Con}(\mathbf{L})$ and the set of 1-filters of $\mathbf{L}$ which is determined by the assignments $S\mapsto \Theta(S)$ and $\theta\mapsto 1/\theta$, where 
$$\Theta(S)=\{(a,b)\in A^{2}\colon a\rightarrow b, b\rightarrow a \in S\}.$$

Let us write $\mathrm{tDLI}_{01}^{+}$ for the class of $\mathrm{tDLI}_{0}^{+}$-algebras whose $\mathrm{DLI}^{+}$-reduct is a DLI$_{1}^{+}$-algebra. In the following we will find a characterization of the congruences of $\mathrm{tDLI}_{01}^{+}$-algebras by means of some particular 1-filters of their $\mathrm{DLI}^{+}$-reducts, namely, the tense 1-filters. This results specialize the results obtained in \cite{CCS} for DLI$_{1}^{+}$-algebras.

\begin{definition}
Let $\mathbf{L}=(\mathbf{A},G,H,F,P)\in \mathrm{tDLI}_{01}^{+}$ and let $S\subseteq A$. We say that $S$ is a tense 1- filter provided that:
\begin{itemize}
\item[(1)] $S$ is a 1-filter, 
\item[(2)] $S$ is closed under $G$ and $H$.
\end{itemize}
\end{definition}

\begin{remark}\label{vale} 
We claim that tense 1- filters of $\mathrm{tDLI}_{01}^{+}$-algebras are closed under $F$ and $P$. Indeed, if $S$ is a 1-filter of a $\mathrm{tDLI}_{01}^{+}$-algebra $\mathbf{L}$ and $x\in S$, then since $G(0)=0$ and $H(0)=0$, by (T4) we get that $G(x)\leq F(x)$ and $H(x)\leq P(x)$. Since $S$ is increasing and closed under $G$ and $H$, then $F(x), P(x)\in S$, as claimed.
\end{remark}

\begin{proposition}\label{congruences 1-filters DLI}
Let $\mathbf{L}=(\mathbf{A},G,H,F,P)\in \mathrm{tDLI}_{01}^{+}$, $\theta\in \mathsf{Con}(\mathbf{L})$ and $S$ be a tense 1-filter of $\mathbf{L}$. Then, the assignments $\theta\mapsto 1/\theta$ and $S\mapsto \Theta(S)$ determine a bijection between $\mathsf{Con}(\mathbf{L})$ and the set of tense 1-filters of $\mathbf{L}$.
\end{proposition}
\begin{proof}

Let $S$ be a tense 1-filter of $\mathbf{L}$. We need to show that $\Theta(S)$ is compatible with all $C\in \{G,H,F,P\}$. So let $(a,b)\in \Theta(S)$. Then $a\rightarrow b,b\rightarrow a\in S$. Observe that  from (T5) $H(a\rightarrow b)\leq P(a)\rightarrow P(b)$. Therefore, since $H(a\rightarrow b)\in S$ and $S$ is increasing, we get $P(a)\rightarrow P(b)\in S$. In a similar fashion it can be proved $P(b)\rightarrow P(a)\in S$. Hence, $(P(a),P(b))\in S$. The proof of the compatibility of $F$ with $\Theta(S)$ is analogue. The proof of the compatibility of $G$ and $H$ with $\Theta(S)$ relies in the employment of (T5).

Finally, if $\theta$ be a congruence of $\mathbf{L}$, it is straightforward to see that  $1/\theta$ is a tense 1-filter of $\mathbf{L}$.
\end{proof}

\section{Congruences of $\mathrm{itKI_{c1}}$}\label{s7}

In this section we consider the variety $\mathrm{itKI_{c1}}$. I.e. all the  $\mathrm{itKI_c}$-algebras such that $x\Rightarrow x=1$. Our aim is to characterize the lattice of congruences of  $\mathrm{itKI_{c1}}$-algebras by means of centered deductive systems. We proceed in several steps. In subsection \ref{congruence tense 1-filters} we apply Theorem \ref{equivalence theorem} to establish an isomorphism between the lattices of congruences of $\mathrm{tDLI}_{01}^{+}$-algebras and $\mathrm{itKI_c}$-algebras (Theorem \ref{iso congruencias via funtor K}). Later, we combine this result and Proposition \ref{congruences 1-filters DLI} to make explicit the correspondence existing between tense 1-filters of $\mathrm{tDLI}_{01}^{+}$-algebras and congruences of $\mathrm{tDLI}_{01}^{+}$-algebras. Afterwards, in subsection \ref{Centered tense deductive systems} we introduce centered deductive systems of $\mathrm{tDLI}_{01}^{+}$-algebras and we prove that there is a bijection between those and the lattice of congruences of  $\mathrm{itKI_{c1}}$-algebras.

\subsection{Congruences of $\mathrm{itKI_c}$ via tense 1-filters}\label{congruence tense 1-filters}

Let $\mathbf{L}$ be a bounded distributive lattice. If $\theta\in \mathsf{Con}(\mathbf{L})$, we can define a congruence $\gamma_{\theta}$ of $K(\mathbf{L})$  by
\[(a,b)\gamma_{\theta}(x,y)\;\text{if and only if}\; (a,x)\in \theta\; \text{and}\; (b,y)\in \theta.\]
Reciprocally, if $\gamma\in \mathsf{Con}(K(\mathbf{L}))$, we can also define a congruence $\theta^{\gamma}$ of $\mathbf{L}$ as 
\[(a,b)\in \theta^{\gamma}\;\text{if and only if}\; (a,0)\gamma (b,0).\]

In Lemma 5.3 of \cite{CCS} it was proved that the assignments $\theta\mapsto \gamma_{\theta}$ and $\gamma\mapsto \theta^{\gamma}$ establish an order isomorphism between $\mathsf{Con}(\mathbf{L})$ and $\mathsf{Con}(K(\mathbf{L}))$. Furthermore, if $\mathbf{A}\in \mathrm{DLI}^{+}$, then the same assignment extends to an order isomorphism between $\mathsf{Con}(\mathbf{A})$ and $\mathsf{Con}(K(\mathbf{A}))$ (Lemma 5.4 and Corollary 5.4 \emph{op.cit.}).

The next two results show that the latter assignment can also be extended for DLI$^{+}_{0}$-algebras and tense centered KI-algebras.

\begin{lemma} \label{iso theta}
Let $\mathbf{L}=(\mathbf{A},G,H,F,P)$ be a tense DLI$^{+}_{0}$-algebra and let $\theta\in \mathsf{Con}(\mathbf{L})$ and $\gamma\in \mathsf{Con}(K(\mathbf{L}))$. Then, $\gamma_\theta\in \mathsf{Con}(K(\mathbf{L}))$ and $\theta^{\gamma}\in \mathsf{Con}(\mathbf{L}).$
\end{lemma}
\begin{proof} 
Let $\theta\in \mathsf{Con}(\mathbf{L})$. We prove that $\gamma_{\theta}\in \mathsf{Con}(K(\mathbf{L}))$. To do so we only need to check that $\theta$ is compatible with $F$ and $G$. Suppose $((a,b),(x,y))\in \gamma_{\theta}$ thus $(a,x)\in \theta$ and $(b,y)\in \theta$. Then it follows that  $(G(a),G(x))\in \theta$ and $(F(b),F(y))\in \theta$, hence $(G(a),F(b))\gamma_{\theta}(G(x),F(y))$ and consequently $G_K(a,b)\gamma_{\theta}G_K(x,y)$, as required. On the other hand, let $\gamma\in \mathsf{Con}(K(\mathbf{L}))$. In order to prove that $\theta^{\gamma}\in \mathsf{Con}(\mathbf{L})$ we shall verify that $\theta^{\gamma}$ is compatible with $G$. Let us assume $(a,b)\in \theta^\gamma$. Then $(a,0)\gamma (b,0)$. By assumption, the latter implies that $G_K(a,0)\gamma G_K(b,0)$. Therefore $(G(a),0)\gamma (G(b),0)$. I.e. $G(a)\theta^{\gamma}G(b)$, as desired.
\end{proof}

\begin{theorem}\label{iso congruencias via funtor K}
Let $\mathbf{L}=(\mathbf{A},G,H,F,P)$ be a tense $DLI^{+}_{0}$-algebra. Then the assignment $f:\mathsf{Con}(\mathbf{L})\rightarrow \mathsf{Con}(K(\mathbf{L}))$ defined by $f(\theta)=\gamma_\theta$ is an order isomorphism.
\end{theorem}

Let $\mathbf{U}=(\mathbf{T},G,H)\in \mathrm{itKI_{c1}}$. Our next goal is to describe explicitly how from a congruence of $\mathbf{U}$ we can get a unique 1-filter of $C(\mathbf{U})$. We start by recalling from Theorem \ref{equivalence theorem}, that the map $\beta_{\mathbf{U}}: \mathbf{U}\rightarrow K(C(\mathbf{U}))$ defined by $\beta_{\mathbf{U}}(x)=(x\vee c, \sim x \vee c)$ is an isomorphism of $\mathrm{itKI_{c1}}$-algebras. Further, from general reasons the map $h_{\mathbf{U}}=\beta_{\mathbf{U}}\times \beta_{\mathbf{U}}$ which is defined by $h_{\mathbf{U}}(x,y)=(\beta_{\mathbf{U}}(x), \beta_{\mathbf{U}}(y))$, induces an isomorphism between $\mathsf{Con}_{\mathrm{itKI}_{c1}}(\mathbf{U})$ and $\mathsf{Con}_{\mathrm{itKI}_{c1}}(KC(\mathbf{U}))$. Now, from Theorem \ref{iso congruencias via funtor K} the map 
\[g:\mathsf{Con}_{\mathrm{itKI}_{c1}}(K(C(\mathbf{U}))) \rightarrow \mathsf{Con}_{\mathrm{tDLI}^{+}_{01}}(C(\mathbf{U}))\] 
defined by $g(\gamma)=\theta^{\gamma}$ is an isomorphism. Consider the composite $w=gh_{\mathbf{U}}$:
\[w:\mathsf{Con}_{\mathrm{itKI}_{c1}}(\mathbf{U})\rightarrow  \mathsf{Con}_{\mathrm{tDLI}^{+}_{01}}(C(\mathbf{U})).\] 

\begin{lemma}\label{congruencias tKLI - tDLF}
Let $\mathbf{U}\in \mathrm{itKI}_{c1}$ and let $\epsilon\in \mathsf{Con}_{\mathrm{itKI}_{c1}}(\mathbf{U})$. Then $\epsilon \cap C(T)^{2}$ is a congruence of $C(\mathbf{U})$.
\end{lemma}
\begin{proof}
From the discussion of above, notice that $w(\epsilon)$ is a congruence of $C(\mathbf{U})$. Now we prove that $w(\epsilon)=\epsilon \cap C(T)^{2}$. Indeed, observe that $(a,b)\in w(\epsilon)$ if and only if $((a,c),(b,c))\in h_{\mathbf{U}}(\epsilon)$. So, since $g$ is an isomorphism, there exist a unique $(x,y)\in \epsilon$ such that $a=x\vee c$, $b=y\vee c$ and $c=\sim x \vee c= \sim y \vee c$. Thus, in particular $\sim x,\sim y \leq c$ and consequently $c\leq x,y$. Hence $x=a$ and $y=b$. Therefore, $(a,b)\in \epsilon \cap C(T)^{2}$, as claimed. The proof of the remaining inclusion is analogue so the result follows.   
\end{proof}

As an immediate application of Proposition \ref{congruences 1-filters DLI}, we get the following:

\begin{corollary}\label{congruencias tKLI - tDLF 1-filtros}
Let $\mathbf{U}\in \mathrm{itKI}_{c1}$ and let $\epsilon\in \mathsf{Con}_{\mathrm{itKI}_{c1}}(\mathbf{U})$. Then $(1/\epsilon) \cap C(T)$ is a tense 1-filter of $C(\mathbf{U})$.
\end{corollary}
\begin{proof}
Straightforward calculations show that $$1/(\epsilon \cap C(T)^{2})=(1/\epsilon) \cap C(T).$$
\end{proof}

Let $\mathbf{U}=(\mathbf{T}, G, H)$ be a member of $\mathrm{itKI}_{c1}$. Now, we proceed to describe explicitly how from a 1-filter of $C(\mathbf{U})$ we can obtain a unique congruence of $\mathbf{U}$. To do so, let $S$ be a tense 1-filter of $C(\mathbf{U})$. Let us write $\mathsf{Fil}_{1}(C(\mathbf{U}))$ for the set of 1-filters of $C(\mathbf{U})$.  Recall that from Proposition \ref{congruences 1-filters DLI} the map 
\[\Theta:\mathsf{Fil}_{1}(C(\mathbf{U}))\rightarrow  \mathsf{Con}_{\mathrm{tDLI}^{+}_{01}}(C(\mathbf{U}))\]
is an isomorphism. By Proposition \ref{iso congruencias via funtor K}, the map  
\[k:\mathsf{Con}_{\mathrm{tDLI}^{+}_{01}}(C(\mathbf{U})) \rightarrow \mathsf{Con}_{\mathrm{itKI}_{c1}}(K(C(\mathbf{U})))\]
defined by $k(\theta)=\gamma_{\theta}$ is a poset isomorphism. Let $(a,b)\in K(C(T))$. Then, since $\beta_{\mathbf{U}}$ is an isomorphism, then from condition (CK) there exists a unique $z\in T$, such that $z\vee c=a$ and $\sim z\vee c=b$. Thus, it is the case that the map $r_{\mathbf{U}}: K(C(\mathbf{U}))\rightarrow \mathbf{U}$ defined by $r_{\mathbf{U}}(a,b)=z$ is in fact $\beta_{\mathbf{U}}^{-1}$ and therefore an isomorphism. By general reasons, the map $w_{\mathbf{U}}=r_{\mathbf{U}}\times r_{\mathbf{U}}$ induces an isomorphism between $\mathsf{Con}_{\mathrm{itKI}_{c1}}(K(C(\mathbf{U})))$ and $\mathsf{Con}_{\mathrm{itKI}_{c1}}(\mathbf{U})$. Consider then the composite $m=w_{T}k\Theta$: 
\[m: \mathsf{Fil}_{1}(C(\mathbf{U})) \rightarrow \mathsf{Con}_{\mathrm{itKI}_{c1}}(\mathbf{U}). \]

\begin{lemma}\label{congruencias tDLF 1-filtros- tKLI}
Let $\mathbf{U}=(\mathbf{T},G,H)\in \mathrm{itKI}_{c1}$ and let $S\in \mathsf{Fil}_{1}(C(\mathbf{U}))$. Then 
{\small
$$\theta_{S}=\{(u,v)\in T^{2}\colon u\Rightarrow (v\vee c), v\Rightarrow (u\vee c), \sim u\Rightarrow (\sim v \vee c),\sim v\Rightarrow (\sim u \vee c)\in S\}$$}
is a congruence of $\mathbf{U}$.
\end{lemma}
\begin{proof}
From the discussion of above, notice that $m(S)$ is a congruence of $\mathbf{U}$. We shall prove that $m(S)=\theta_{S}$. To do so, observe first that from Theorem \ref{iso congruencias via funtor K} and Proposition \ref{congruences 1-filters DLI} 

{\small
\begin{equation}\label{equation 1}
k\Theta(S)=\{((a,b),(x,y))\in K(C(T))^{2}\colon a\Rightarrow x,x\Rightarrow a, b\Rightarrow y,y\Rightarrow b\in S\}.
\end{equation}}

Thus if we write $r_{\mathbf{U}}(a,b)=u$ and $r_{\mathbf{U}}(x,y)=v$, it is the case that $u\vee c=a$, $v\vee c=x$, $\sim u\vee c=b$ and $\sim v\vee c=y$. Since $c\leq v\vee c $, from Lemma \ref{monotonia}, $c\Rightarrow (v\vee c)=1$. So, from (I2) we get:
\begin{displaymath}
\begin{array}{rcl}
a\Rightarrow x & = & (u\vee c)\Rightarrow (v\vee c)
\\
 & = & (u\Rightarrow (v\vee c))\wedge (c\Rightarrow (v\vee c))
\\
  & = & (u\Rightarrow (v\vee c))\wedge 1
\\  
  & = & u\Rightarrow (v\vee c).
\end{array}
\end{displaymath}
By the same arguments of above, we can prove also that $x\Rightarrow a=v\Rightarrow (u\vee c)$, $b\Rightarrow y=\sim u\Rightarrow (\sim v\vee c)$ and $y\Rightarrow b=\sim v\Rightarrow (\sim u\vee c)$. Then
{\small
\[m(S)=\{(u,v)\in T^{2}\colon u\Rightarrow (v\vee c), v\Rightarrow (u\vee c), u\Rightarrow (\sim v\vee c), \sim v\Rightarrow (\sim u\vee c)\in S\}.\]}
Hence $m(S)=\theta_{S}$, as claimed.
\end{proof}

\subsection{Centered tense deductive systems}\label{Centered tense deductive systems}

\begin{definition}
Let  $\mathbf{U}=(\mathbf{T}, G, H)$ be $\mathrm{itKI}_{c1}$-algebra. A subset $D$ of $T$ is called a tense deductive system provided that:
\begin{itemize}
    \item[(tD1)] $1\in D$,
    \item[(tD2)] If $u,u\Rightarrow v \in D$ then $v\in D$,
    \item[(tD3)] $G(u),H(u)\in D$.
    \end{itemize}
    We write $\mathsf{tD}(\mathbf{U})$ for the set of tense deductive systems of $\mathbf{U}$.
\end{definition}

\begin{remark}\label{Deductive systems are increasing}
Let  $\mathbf{U}=(\mathbf{T}, G, H)$ be $\mathrm{itKI}_{c1}$-algebra. Observe that every tense deductive system $D$ of $\mathbf{U}$ is non-empty and increasing. Indeed, the non-emptiness is granted from (tD1) and if $x\leq y$ and $x\in D$, then it is the case that $1=y\Rightarrow y\leq x\Rightarrow y$, so $x\Rightarrow y\in D$. Hence, from (tD2) we conclude $y\in D$.
\end{remark}

Now consider the following term 
\[t(x,y,z):=((x\wedge z)\Rightarrow y)\Rightarrow (x \Rightarrow y).\]

\begin{lemma}\label{centered deductive systems are 1-filters}
Let  $\mathbf{U}=(\mathbf{T}, G, H)$ be a $\mathrm{itKI}_{c1}$-algebra and let $D\in \mathsf{tD}(\mathbf{U})$. Then $S_{D}=D\cap C(T)$ is a tense 1-filter of $C(\mathbf{U})$ if and only if for every $u,v\in D$ and $a,b\in C(T)$ the following conditions hold:
\begin{itemize}
 \item[{\rm (tD4)}] $(u\wedge v)\vee c\in D$,
 \item[{\rm (tD5)}] $t(a,b,u\vee c)\in D$.
\end{itemize}
\end{lemma}
\begin{proof}
On the one hand, let us assume that $S_{D}$ is a 1-filter of $C(\mathbf{U})$. If $u,v\in D$, then because of Remark \ref{Deductive systems are increasing}, it is the case that $u\vee c,v\vee c\in S_{D}$. Therefore $(u\wedge v)\vee c=(u\vee c)\wedge (v\vee c)\in S_{D}$ and consequently (tD4) holds. Now, let $a,b\in C(T)$ and $u\in D$. Again from Remark \ref{Deductive systems are increasing} we get $u\vee c\in S_{D}$ thus $t(a,b,u\vee c)\in S_{D}$ so (tD5) holds. On the other hand, let us assume $D$ is a tense deductive filter of $(\mathbf{T}, G, H)$ satisfying (tD4) and (tD5). From Remark \ref{Deductive systems are increasing} we get that $S_{D}$ is increasing. Now, let $x,y\in S_{D}$, then from (tD4) $(x\wedge y)\vee c=x\wedge y\in D$ so $x\wedge y\in S_D$. I.e. $S_D$ is a lattice filter. It is clear that $1\in S_D$. Let $a,b\in C(T)$ and $u\in S_D$. Due to $u\in D$ and $D$ is increasing by Remark \ref{Deductive systems are increasing}, we get $u\vee c\in S_D$. Hence, since $C(T)$ is closed under $\Rightarrow$, we obtain $t(a,b,u)\in C(T)$ and by (tD5) $t(a,b,u\vee c)=t(a,b,u)\in D$. Finally, if $x\in S_D$, then $c\leq x$ and $x\in D$. Thus, from Proposition \ref{propiedades tKLc}, and (tD3) we obtain that $H(x)\in S_D$. The proof for $G$ is analogue. I.e. $S_D$ is a tense 1-filter of $C(\mathbf{U})$ as required. This concludes the proof.
\end{proof}

\begin{definition}\label{centered tense d.s.}
Let  $\mathbf{U}=(\mathbf{T}, G, H)$ be $\mathrm{itKI}_{c1}$-algebra. We say that a tense deductive system $D$ of $\mathbf{U}$ is a centered deductive filter provided that: 
\begin{itemize}
\item[(ctD1)] $S_D$ is a tense 1-filter of $C(\mathbf{\mathbf{U}})$,
\item[(ctD2)] For every $u\in T$, if $\sim u \Rightarrow c, 1\Rightarrow (u\vee c) \in D$, then $u\in D$.
\end{itemize}
\end{definition}

We write $\mathsf{tD}_c(\mathbf{U})$ for the set of tense deductive systems of $\mathbf{U}$.

\begin{lemma}\label{1-filtros tDLF - sdtc tKLI}
Let $\mathbf{U}=(\mathbf{T},G,H)\in \mathrm{itKI}_{c1}$ and let $S\in \mathsf{Fil}_{1}(C(\mathbf{U}))$. Then, 
\[
D_{S}=\{u\in T\colon \sim u\Rightarrow c, 1\Rightarrow (u\vee c)\in S\}.
\]
is a tense deductive system of $\mathbf{U}$ such that $D_{S}\cap C(T)=S$. Furthermore, $D_{S}$ is a centered tense deductive system of $\mathbf{U}$.
\end{lemma}
\begin{proof}

We proceed to show that $D_{S}$ satisfies conditions $(tD1)$, $(tD2)$ and $(tD3)$:
\noindent
\\ 
$(tD1)$ Observe that $1\Rightarrow (1\vee c)=1\Rightarrow 1=1\in S$ and since $\sim 1=0\leq c$, then $\sim 1\Rightarrow c =1 \in S$. 
\\
$(tD2)$ Suppose $u,u\Rightarrow v\in D_{S}$. In order to prove that $v\in D_{S}$ we need to show that $\sim v\Rightarrow c$ and $1\Rightarrow (v\vee c)$ are elements of $S$. Since $u\wedge (u\Rightarrow v)\leq v$, we get $\sim v\leq \sim u \vee \sim(u\Rightarrow v)$. Hence by Lemma \ref{monotonia} and (I2),
\[(\sim u \Rightarrow c)\wedge (\sim(u\Rightarrow v)\Rightarrow c)\leq \sim v \Rightarrow c.\]  
Therefore, since $\sim u \Rightarrow c$ and $\sim(u\Rightarrow v)\Rightarrow c$ belong to $S$ by assumption and $S$ is a lattice filter by definition, then we conclude $\sim v\Rightarrow c\in S$. Finally, by (KI2), $(u\wedge (u\Rightarrow v))\vee c\leq v\vee c$, so it follows $(u\vee c)\wedge ((u\Rightarrow v)\vee c)\leq v\vee c$. So, from Lemma \ref{monotonia} and (I1) it is the case that 
\[(1\Rightarrow (u\vee c))\wedge (1\Rightarrow ((u\Rightarrow v)\vee c)\leq 1\Rightarrow(v\vee c).\]
Hence, since $1\Rightarrow (u\vee c)$ and $1\Rightarrow ((u\Rightarrow v)\vee c)$ are elements of $S$ by assumption and $S$ is a lattice filter, we obtain $1\Rightarrow (v\vee c)\in S$, as desired.
\\
$(tD3)$ If $u\in D_{S}$, it follows that $\sim u\Rightarrow c, 1\Rightarrow (u\vee c)\in S$. Then, $G(\sim u\Rightarrow c)\leq F(\sim u)\Rightarrow F(c)=\sim G(u)\Rightarrow c$ by (t6). Therefore, since $G(\sim u\Rightarrow c)\in S$ and $S$ is increasing, $\sim G(u)\Rightarrow c\in S$. Finally, $G(1\Rightarrow (u\vee c))\leq G(1)\Rightarrow G(u\vee c)=1\Rightarrow G(u)\vee c$ by (t5) and Proposition \ref{propiedades tKLc}. Thus, since $S$ is increasing, $1\Rightarrow G(u)\vee c\in S$ as claimed. The proof that $D_{S}$ is closed by $H$ is similar.    
\\

Now we prove $D_{S}\cap C(T)=S$. To do so, let $u\in S$. We need to show that $\sim w\Rightarrow c,1\Rightarrow (w\vee c)\in S$. Observe that due to $c\leq w$, we get $\sim w\Rightarrow c=1\in S$ so the latter reduces to prove $1\Rightarrow w\in S$. Since $S$ is 1-filter of $C(\mathbf{U})$ by assumption,  $t(1,w,w)=1\Rightarrow (1\Rightarrow w)\in S$. Thus, because $t(1,w,w)\leq 1\Rightarrow w$, and $S$ is increasing, we conclude $1\Rightarrow w\in S$. The proof of the remaining inclusion follows from the fact that $1\Rightarrow w\in S$ and $1\Rightarrow w\leq w$. Hence $D_S$ satisfies (ctD2). Finally, in order to prove that $D_S$ satisfies (ctD2), let $u\in T$ and suppose that $\sim u\Rightarrow c, 1\Rightarrow (u\vee c)\in D_S$. Notice that since $\sim u\Rightarrow c, 1\Rightarrow (u\vee c)\in C(T)$, then \[1=\sim(\sim u\Rightarrow c)\Rightarrow c=\sim(1\Rightarrow (u\vee c))\Rightarrow c.\]
So the assumptions about $u$ reduce to $1\Rightarrow (\sim u\Rightarrow c)$, and $1\Rightarrow (1\Rightarrow (u\vee c))$ belong to $S$. But since $S$ is increasing we get $\sim u\Rightarrow c$ and $1\Rightarrow (u\vee c)$ are elements of $S$. I.e. $u\in D_S$. This concludes the proof.
\end{proof}

\begin{lemma}\label{ctds to ctsd}
Let $\mathbf{U}\in \mathrm{itKI}_{c1}$ and let $D$ be a centered tense deductive system of $\mathbf{U}$. Then $D_{S_D}=D$. 
\end{lemma}
\begin{proof}
On the one hand, let $u\in D_{S_D}$. Then it follows that $\sim u\Rightarrow c, 1\Rightarrow (u\vee c)\in D$. Hence, by (ctD2), $u\in D$. On the other hand, let $u\in D$. By (KI4), $(\sim u\Rightarrow c)\wedge c=(u\vee c)\wedge c=c$ so $\sim u\Rightarrow c\in C(T)$. Since $C(T)$ is closed under $\Rightarrow$, then $1\Rightarrow (u\vee c)\in C(T)$. Further, from Lemma \ref{centered deductive systems are 1-filters} (2), $t(1,u\vee c,u\vee c)=1\Rightarrow (1\Rightarrow (u\vee c))\in D$. Thus, since $D$ in increasing by Remark \ref{Deductive systems are increasing} and $1\Rightarrow (1\Rightarrow (u\vee c))\leq 1\Rightarrow (u\vee c)$, then $1\Rightarrow (u\vee c)\in D$. Notice that $c\leq \sim u \vee c$, then by (I2):
\begin{equation}\label{equation 4}
(\sim u \vee c)\Rightarrow c= \sim u\Rightarrow c.
\end{equation}
Since $D$ is increasing, $u\vee c\in D$. Then from Lemma \ref{centered deductive systems are 1-filters} (1) we obtain $t(\sim u\vee c,c,u\vee c)\in D$. Now observe that since $\mathbf{U}$ has a centered Kleene algebra reduct, $\sim u\wedge u\leq c$. Then 
\[((\sim u\vee c)\wedge (u\vee c))\Rightarrow c= ((\sim u\wedge u)\vee c)\Rightarrow c =1.\]
Therefore, from (\ref{equation 4}): 
\[t(\sim u\vee c,c,u\vee c)=1\Rightarrow ((\sim u\vee c)\Rightarrow c)=1\Rightarrow (\sim u\Rightarrow c)\in D.\]
Hence, due to $1\Rightarrow (\sim u\Rightarrow c)\leq \sim u\Rightarrow c$ and because $D$ is increasing by Remark \ref{Deductive systems are increasing}, we can conclude that $\sim u\Rightarrow c\in D$. I.e. $u\in D_{S_D}$.
\end{proof}

\begin{theorem}\label{congruences ctds}
Let $\mathbf{U}\in \mathrm{itKI}_{c1}$. Then, the maps $S\mapsto D_S$ and $D\mapsto S_D$ establish an order isomorphism between the lattice of 1-filters of $C(\mathbf{U})$ and the lattice of centered deductive systems of $\mathbf{U}$.  
\end{theorem}
\begin{proof}
From Lemmas \ref{1-filtros tDLF - sdtc tKLI} and \ref{ctds to ctsd} we get that $S_{D_S}=S$ and $D_{S_D}=D$. Straightforward calculations prove that both assignments are monotone.
\end{proof}

Therefore, as a consequence of Proposition \ref{congruences 1-filters DLI} and Theorem \ref{congruences ctds} we can conclude:

\begin{corollary}
Let $\mathbf{U}\in \mathrm{itKI}_{c1}$. Then, there the lattices $\mathsf{tD}_c(\mathbf{U})$ and $\mathsf{Con}_{\mathrm{itKI}_{c1}}(\mathbf{U})$ are isomorphic.
\end{corollary}

\section{Kalman's construction for tense Heyting algebras}\label{s8}

Recall that an algebra ${\bf A}=\langle A,\wedge,\vee,\to,0,1\rangle$ of type $(2,2,2,0,0)$ is a Heyting algebra  if the following conditions hold: $x\wedge y\leq z\Longleftrightarrow x\leq y\to z.$ Such a condition is known as residuation.

The following lemma gives us a characterization of Heyting algebras in terms of DLI$^{+}$-algebras.

\begin{lemma}\label{DLI que son Heyting}
Let $\mathbf{L}$ be a ${\rm DLI}_{1}^{+}$-algebra. Then $\mathbf{L}$ is a Heyting algebra if and only if $x\leq y\rightarrow x$ for every $x,y\in L$.
\end{lemma}
\begin{proof}
In order to proof our claim we need to show that $x\wedge y\leq z$ if and only if $x\leq y\rightarrow z$. Indeed: let $x\wedge z\leq y$. Then, $x\to (x\wedge z)\leq x\to y$. From (I1), we have that $(x\to x)\wedge (x\to z)\leq x\to y$. Since $x\to x=1$, we obtain $x\to z\leq x\to y$. Then, from (I6), $z\leq x\to y$. The converse holds for  (I5). This concludes the proof.
\end{proof}

\begin{definition}\label{NelsonKI} We say that an $\mathrm{itKI_{c1}}$-algebra {\bf T} is a Nelson $\mathrm{itKI_{c1}}$-algebra if the DLI$_{1}$-reduct of {\bf T} satisfies: 
\begin{itemize}
\item[(I6)] $x\leq y\Rightarrow x.$
\end{itemize}
\end{definition}

It is immediate from Definition \ref{NelsonKI} that Nelson $\mathrm{itKI_{c1}}$-algebra satisfy (KI3).

\begin{lemma}\label{NKLI a Heyting}
Let $\mathbf{T}$ be a Nelson $\mathrm{itKI_{c1}}$-algebra. Then $C(\mathbf{T})$ is a Heyting algebra.
\end{lemma}
\begin{proof}
Notice that from Lemma \ref{DLI que son Heyting} we only need to show that $C(\mathbf{T})$ satisfies (I5) and (I6), but this follows from Lemma \cite[Proposition 4.7]{CCS} and Definition \ref{NelsonKI}.
\end{proof}

We denote by $\mathsf{HA}$ the category of Heyting algebras and homomorphisms. M. Fidel \cite{F} and D. Vakarelov \cite{V} proved independently that if $\mathbf{A}$ is a Heyting algebra, then the Kleene algebra $K(\mathbf{A})$ is a centered Nelson algebra, in which the weak implication is defined for pairs $(a,b)$ and $(d,e)$ in $K(A)$ as follows:

$$(a,b)\Rightarrow (d,e):=(a\to d,a\wedge e).$$

\begin{lemma}\label{Heyting a NKLI}
Let $\mathbf{A}$ be a Heyting algebra. Then $K(\mathbf{A})$ is a Nelson $\mathrm{itKI_{c1}}$-algebra.
\end{lemma}
\begin{proof}
Observe that from Definition \ref{NelsonKI}, we need to prove that the ${\rm DLI}_{1}$-reduct of $K(\mathbf{A})$ satisfies (I6). Indeed, since $x\wedge y=0$, we get $x\wedge y\leq b$, so $x\leq y\to b$. Moreover, by Lemma \ref{DLI que son Heyting} we have $x\leq a\to x$. Then, $x\leq (a\to x)\wedge (y\to b)$. On the other hand, $a\wedge y\leq y$. Hence, \[(x,y)\leq((a\to x)\wedge (y\to b),a\wedge y)=(a,b)\Rightarrow (x,y),\]
as claimed.
\end{proof}

Let $\mathsf{NKI}$ the category of Nelson $\mathrm{itKI_{c1}}$-algebra and homomorphisms and let $\mathsf{Nel}_{c}$ the category of centered Nelson algebras with homomorphisms. Then, from Lemmas \ref{NKLI a Heyting} and \ref{Heyting a NKLI}, Theorem 2.9 and Theorem 4.12 of \cite{CCS} we obtain:

\begin{corollary}\label{c8.7} 
The functors K and C restrict to an equivalence between $\mathsf{NKI}$ and $\mathsf{HA}$ with natural isomorphisms $\alpha$ and $\beta$. Moreover, $\mathsf{NKI}$ is also equivalent to $\mathsf{Nel}_{c}$.
\end{corollary}

Now we are going to extend the results above for the case of tense Nelson  $\mathrm{itKI_{c1}}$ and tense Heyting algebras.

\begin{definition}Let $\mathbf{A}=\langle A,\vee,\wedge,\to,0,1\rangle$ be a Heyting algebra, let $G$, $H$, $F$ and $P$ be unary operations on $A$ satisfying:

\begin{itemize}
\item [(T0)] $G(0)=0$ and $H(0)=0,$
    \item [(T1)] $P(x)\leq y$ if and only if $x\leq G(y),$
    \item [(T2)] $F(x)\leq y$ if and only if $x\leq H(y),$
    \item [(T3)] $G(x)\wedge F(y)\leq F(x\wedge y)$ and $H(x)\wedge P(y)\leq P(x\wedge y),$
    \item [(T4)] $G(x\vee y)\leq G(x)\vee F(y)$ and $H(x\vee y)\leq H(x)\vee P(y).$
\end{itemize}
Then the algebra $\mathbf{L}=(\mathbf{A},G,H,F,P)$ will be called tense Heyting algebra.
\end{definition}

\begin{proposition}\label{HDLI} Let $\mathbf{L}=(\mathbf{A},G,H,F,P)$ be a tense Heyting algebra. Then, $\mathbf{L}\in \mathrm{tDLI}_{01}^{+}$.
\end{proposition}

\begin{proof} 
 From the
results established in Lemma \ref{NelsonKI} it only remains to prove that conditions (T5) and (T6) in Definition \ref{tDLI} hold.

\noindent (T5):  Taking into account that $G$ is increasing, we obtain that $G(x\wedge (x\to y))\leq G(y).$ So, $G(x)\wedge G(x\to y)\leq G(y)$. Therefore, $G(x\to y)\leq G(x)\to G(y)$.

\noindent (T6): From (T3), we have that $G(x\to y)\wedge F(x)\leq F(x\wedge (x\to y))$. Since, $F$ is increasing, we obtain $G(x\to y)\wedge F(x)\leq F(y)$. Hence, $G(x\to y)\leq F(x)\to F(y).$

\end{proof}

\begin{definition} A tense $\mathrm{itKI_{c1}}$-algebra ${\bf U}=({\bf T},G,H)$ is a tense Nelson $\mathrm{itKI_{c1}}$-algebra if the reduct ${\bf T}$ is a Nelson $\mathrm{itKI_{c1}}$-algebra. 
\end{definition}

\begin{lemma} Let ${\bf L}=({\bf A},G,H,F,P)$ be a tense Heyting algebra. Then, $$K({\bf L})=({\bf A}_{K},G_K,H_K)$$ is a tense Nelson $\mathrm{itKI_{c1}}$-algebra.
\end{lemma}

\begin{proof} This follows from Lemma \ref{Heyting a NKLI}, Proposition \ref{HDLI} and Lemma \ref{kalmanH}. 
\end{proof}

\begin{lemma} Let ${\bf U}=({\bf T},G,H)$ be a tense Nelson $\mathrm{itKI_{c1}}$-algebra. Then, $$C({\bf T})=({\bf T}_{C},G,H,F,P)$$ is a tense Heyting algebra.
\end{lemma}

\begin{proof} This follows from Lemma \ref{tDLI+0 to tKIc} and Lemma \ref{Heyting a NKLI}.
\end{proof}

 In what follows, we denote by $\mathsf{tHA}$ the category of tense Heyting algebras and homomorphisms and let $\mathsf{tNKI}$ the category of tense Nelson $\mathrm{itKI_{c1}}$ and homomorphisms.

We conclude by noticing that the following result follows from standard calculations based on previous results of this section.

\begin{theorem} The functors {\rm K} and {\rm C} establish a categorical equivalence between $\mathsf{tHA}$ and $\mathsf{tNKI}$ with natural isomorphisms $\alpha$ and $\beta$.
\end{theorem}

\section*{Acknowledgements}

The authors acknowledge many helpful comments from the anonymous referee, which considerably
improved the presentation of this paper. Gustavo Pelaitay want to thank the institutional support of Consejo Nacional de Investigaciones Cient\'ificas y T\'ecnicas (CONICET).



\AuthorAdressEmail{Gustavo Pelaitay}{CONICET and Instituto de Ciencias B\'asicas\\
Universidad Nacional de San Juan\\
5400, San Juan, Argentina}{gpelaitay@gmail.com}


\AdditionalAuthorAddressEmail{William Zuluaga}{ Departamento de Matem\'atica\\ Facultad de Ciencias Exactas\\
Universidad Nacional del Centro\\
 Tandil, Argentina}{wizubo@gmail.com}
\end{document}